\documentclass[11pt,leqno]{article}
\usepackage{amsmath,amssymb,amsthm,nicefrac,mathrsfs}
\usepackage[colorlinks]{hyperref}
\usepackage{tabu}
\usepackage{bbm}
\usepackage{graphics}
\usepackage{verbatim}
\usepackage[a4paper, margin=4cm, footskip=1cm]{geometry}
\usepackage{caption}
\usepackage[T1]{fontenc}
\usepackage{booktabs}
\usepackage{siunitx}
\newtheorem{theorem}{Theorem}[section]
\newtheorem{definition}[theorem]{Definition}
\newtheorem{lemma}[theorem]{Lemma}

\newtheorem{corollary}[theorem]{Corollary}

\newtheorem{remark}[theorem]{Remark}

\numberwithin{equation}{section}

\newcommand{\be}{\begin{equation}}
	\newcommand{\ee}{\end{equation}}

\newcommand{\bes}{\begin{equation*}}
	\newcommand{\ees}{\end{equation*}}

\providecommand{\abs}[1]{\vert#1\vert}
\providecommand{\norm}[1]{\Vert#1\Vert}
\providecommand{\pp}[1]{\langle#1\rangle}

\newcommand{\N}{\mathbf{varrho}}

\renewcommand{\geq}{\geqslant}
\renewcommand{\leq}{\leqslant}

\def\m1{\mathbf{1}}

\def\D{\mathbb{D}}

\def\N{\mathbb{N}}
\def\H{\mathcal{H}}

\newcommand\bb[1]{\left[#1\right]}
\newcommand\set[1]{\left\{#1\right\}}

\allowdisplaybreaks[1]

\allowdisplaybreaks[1]
\author{Ngartelbaye Guerngar\\
	University of North Alabama\\
	\and James McCormick\\ University of North Alabama\\}
\title{Distributed-order space-time fractional diffusions in bounded domains
	\date{}
}

\begin{document}
	
	\maketitle
	\begin{abstract}
	We provide explicit classical solutions and stochastic analogues for distributed-order space-time fractional diffusion equations on bounded domains with zero exterior boundary conditions. We also show that our results still hold when the mixing measure  in the distributed-order time-derivative is singular.

	\end{abstract}
	\newpage
	\section{Introduction}
	
	It is well known that fractional derivatives can be used to model time delays in a diffusion process. When the order of the fractional derivative is distributed over the unit interval, it is useful in modeling a mixture of delay sources. There is a plethora of literature devoted to the topic.
	


	 Einstein in his paper  \cite{Einsten} provided a mathematical model linking random walks , the diffusion equation $\frac{\partial u}{\partial t} =\Delta u$ and the Brownian motion. In fact the scaling limits of a simple and centered random walk with finite variance jumps yields a Brownian motion as follows \cite{ChenMeerNane}:
	recall that a random walk $S_t=Y_1+Y_2+\cdots+Y_{[t]}$ is a sum of independent and identically distributed (iid) $\mathbb{R}^n$-valued random vectors. Here $[t]$ is the largest integer not exceeding $t$ and $S_t$ represents the location of a random particle at time $t.$ Suppose the distribution of $Y$ is spherically symmetric with $\mathbb{E}\big[Y_1\big]=0$ and $\mathbb{E}\big[|Y_1|^2\big]=:\sigma^2<\infty.$   Then by Donsker's invariance principle, the random process $\big\{\lambda^{-1/2}S_{\lambda t},  t\geq 0\big\}$ converges weakly,  as $\lambda\rightarrow\infty$, in the Skorohod space to Brownian motion $\big\{B_t, t\geq 0\big\}$ with  $\mathbb{E}\big[|B_1|^2\big]=\sigma^2.$
	However, if the step random variable $Y_1$ is spherically symmetric  with $\mathbb{P}\big(|Y_1|>x\big)\sim C x^{-\alpha} $ as $x\rightarrow\infty$ for some $0<\alpha<2$ and $C>0$, then $\mathbb{E}\big[|Y_1|^2\big]=\infty$ and the extended Central Limit theorem shows that $\big\{\lambda^{-1/\alpha}S_{\lambda t},  t\geq 0\big\}$
	converge weakly to a rotationally symmetric $\alpha$-stable L\'evy motion $\big\{X_t, t\geq 0\big\}$ with  $\mathbb{E}\big[e^{i\xi\cdot X_t}\big]=e^{-C_0|\xi|^\alpha}$ for every $\xi\in\mathbb{R}^d$ and $t\geq 0$ where the constant $C_0$ depends only on $\alpha, C$ and $n$ \cite{MeerSchef}.
	The probability densities of the Brownian motion variables solve a diffusion equation and thus we refer to the Brownian motion as the stochastic solution to the aforementioned diffusion equation.
	Now if we impose a random waiting time $T_n$ before the $n^{th}$ random walk, then the position of the particle, then $S_n$ represents the the position of the particle at time $T_n=J_1+J_2+\cdots+J_n.$ Since the number of jumps by time $t>0$ is $N_t=\max\{n: \ T_n\leq t\}$, the position of the particle at time $t>0$ is the time-changed process $S_{N_t}$. 
	
	\vspace{0.5em}
	The space-time fractional diffusion $\frac{\partial ^\beta u}{\partial t^\beta}=-(-\Delta)^{\alpha/2} u$ ; where the time-fractional
	derivative is based on the Caputo fractional derivative of order $0<\beta<1$  is used to
	model anomalous sub-diffusions in which a cloud of particles spread slower than square
	root of time \cite{ChenMeerNane}. 
	Initial-value problems for the muti-term time-fractional operator $\sum\limits_{j=1}^nc_j\frac{\partial^{\beta_j}}{\partial t^{\beta_j}}$, for $0<\beta_1<\beta_2<\cdots<\beta_j<1,$ are of great interest due to their vast capability of modeling the anomalous diffusion phenomena in highly heterogeneous aquifer and complex viscoelasticity material, see for example \cite{LiLiuMoto,LiMoto,MijNane} and the references therein.
	
	\vspace{0.5em}

	  Distributed-order diffusions are
	used to model ultra-slow diffusion where a plume of particles spread at a logarithmic
	rate of time \cite{LiKoTo}, which can be useful in particle tracking. 
	In \cite{Nab}, the author  studied the eigenvalue problem for  the distributed-order  time-fractional diffusion equation $\mathbb{D}^{(\nu)}u=c\frac{\partial^2}{\partial x^2}u$  under several boundary conditions. The asymptotic behavior of the solutions for the distributed-order time-fractional initial-boundary-value problems in bounded multi-dimensional domains was studied in \cite{LiKoTo}. The paper \cite{MijNane} studied strong analytical solutions to the distributed-order fractional diffusion equation $\mathbb{D}^{(\nu)}u=Lu$, where $L$ is the generator of a uniformly bounded, strongly continuous semigroup, in a sectorial region of the complex plane. The authors in \cite{MeerNaneVella},  develop explicit strong solutions for the distributed-order time-fractional equation $\mathbb{D}^{(\nu)}u=Lu$, where $L$ is an uniformly elliptical operator on a bounded subset of $\mathbb{R}^n$ with Dirichlet boundary conditions. A survey of recent results on the fractional Cauchy problems and its generalization to a bounded subset of euclidean space was discussed in \cite{Nane}.
	
		Distributed order fractional derivatives are also connected to random walks limits as follows: for each $c>0,$ take a sequence of i.i.d. waiting time $(J_n^c)$ and i.i.d jumps $Y_n^c$. Let $X^c(n)=Y_1^c+Y_2^c+\cdots+Y_n^c$ be the particle location after $n$ jumps, and $T^c(n)=J_1^c+J_2^c+\cdots J_n^c$ the time of $n^{th}$ jump. Suppose $X^c(cu)\implies A(t)$ and $T^c(cu)\implies W_t$ as $c\rightarrow\infty,$ the limit $A(t)$ and $W_t$ are independent L\'evy process. the number of jumps by time $t>0$ is $N_t^c=\max\{n: \ T(n)^c\leq t\}$, and \cite[Theorem 2.1]{MeerSchef2} shows that $X^c(N_t^c)\implies A(E_t)$, where 
	\begin{equation}\label{eq:Et}
	E_t=\inf\{\tau: W_\tau>t\}.
	\end{equation}
In applications to finance for example, the waiting times $J_i^c$ represent the times between transactions and the jumps $Y_i^c$ are the prices jumps (or log-returns)\cite{MeerSchef2}.

\vspace{0.5em}
In this article, we develop explicit classical solutions for the distributed-order space-time fractional diffusion equations $\mathbb{D}^{(\nu)}u=-(-\Delta)^{\alpha/2}$ on bounded domains $D\subset \mathbb{R}^n$, with zero exterior Dirichlet boundary conditions and identify the underlying stochastic process. This provides an extension of works in  \cite{ChenMeerNane, MeerNaneVella,MeerNaneVella2,MijNane}. To the best of our knowledge, this is the first article to investigate such problem.

\vspace{0.5em}
The rest of the article is organized as follows: in Section 2, we review some background results then in Section 3 we state and prove our main results. In section 4, we study the special case when the mixture measure is singular.
	\section{Preliminaries}
	
	Following \cite{Caputo}, the Caputo fractional derivative is defined for $0<\beta<1$ as
	\begin{equation} \label{eq:caputo}
		\frac{\partial^\beta u(t,x)}{\partial t^\beta} = \frac{1}{\Gamma(1-\beta)} \int_0^t \frac{\partial u(r,x)}{\partial r} \frac{dr}{(t-r)^\beta}.
	\end{equation}
	Its Laplace transform is given by
	\begin{equation} \label{eq:caputo laplace}
		\int_0^\infty e^{-st} \frac{\partial^\beta u(t,x)}{\partial t^\beta} ds = s^\beta \tilde{u}(s,x) - s^{\beta-1} u(0,x).
	\end{equation}
	One nice feature of the Caputo fractional derivative is that (as shown above) its Laplace transform incorporates the initial value in the same way as does the first derivative in contrast to its other fractional derivative counterpart, the Riemann-Louiville fractional derivative. Next,
	the distributed order fractional derivative is defined as
	\begin{equation} \label{eq:dofd1}
		\D^{(\nu)} u(t,x): = \int_0^1 \frac{\partial^\beta u(t,x)}{\partial t^\beta} \nu(d\beta)
	\end{equation}
	where $\nu$ is a finite Borel measure with $\nu(0,1) > 0$.

	
	A specific mixture model from \cite{MeerSchef} gives rise to distributed order fractional derivatives as follow: let $\big(B_i, 0<B_i<1\big)_i$ be i.i.d random variables satisfying
	
	\begin{equation}\label{PJ}
	\mathbb{P}\big(J_i^c>u|B_i=\beta\big)=\begin{cases}
	1& \ \text{if} \ \ 0\leq u< c^{-1/\beta}\\
	c^{-1}u^{-\beta}& \ \text{if} \  \ u\geq c^{-1/\beta}.
	\end{cases}
	\end{equation}
Then $T^c(cu)\implies W_t$, a subordinator with 
\begin{equation}\label{eq:LapExp}
\mathbb{E}\big[e^{-sW_t}\big]=e^{-t\psi_W(s)},
\end{equation}
where
	\begin{equation} \label{eq:LExp}
		\psi_W(s) = \int\limits_0^\infty \big(e^{-sx}-1\big) \phi_W(dx)
	\end{equation}
	
The associated L\'evy measure is known to be 
\begin{equation}\label{eq:Lmeas}
\phi_W(t,\infty)=\int\limits_0^1t^{-\beta}\mu(d\beta),
\end{equation}
where $\mu$ is the distribution of $B_i$. In light of identity \eqref{eq:Lmeas}, a simple calculation shows that the L\'evy exponent \eqref{eq:LExp} can be expressed as 
	
		\begin{equation} \label{eq:psi_w}
		\psi_W(s) = \int_0^1 s^\beta \Gamma(1-\beta) \mu(d\beta)
	\end{equation}
Recall from \cite[Theorem 3.10]{MeerSchef1} that $c^{-1}N_t^c\implies E_t$, where $\{E_t\}_{t\geq 0}$ is the hitting time process of  the subordinator $\{W_u\}_{u\geq 0}$ defined by \eqref{eq:Et}. 
	
It is well known that the L\'evy process $A(t)$ defines a strongly continuous convolution semigroup $\big\{p_t, t\geq 0\big\}$ with generator $-(-\Delta)^{\alpha/2}$, and $A(E_t)$ is the stochastic solution to the  distributed order fractional diffusion equation 
\begin{equation}\label{eq:Main}
\mathbb{D}^{(\nu)}u=-(-\Delta)^{\alpha/2}u,
\end{equation}
where $\mathbb{D}^{(\nu)}$ is defined in \eqref{eq:dofd1} with the mixing measure $\nu$ given by
\begin{equation}\label{eq:MixM}
\nu(d\beta)=\Gamma(1-\beta)\mu(d\beta).
\end{equation}
 We further impose the condition 
\begin{equation}\label{eq:ExC}
\int\limits_0^1\frac{1}{1-\beta}\mu(d\beta)<\infty
\end{equation}
to ensure that $\nu(0,1)<\infty$. Since $\phi_W(0,\infty)=\infty$ in \eqref{eq:Lmeas}, it follows from \cite[Theorem 3.1]{MeerSchef} that $E_t$ has a Lebesgue density

\begin{equation}\label{eq:LebDens}
g(t,x)=\int\limits_0^t\phi_W(t-y,\infty)P_{W_x}(dy).
\end{equation}

Here, for $0<\alpha\leq 2,$ $-(-\Delta)^{\alpha/2}$ is defined for 
\begin{equation*}
h\in \text{Dom}\big(-(-\Delta)^{\alpha/2}\big):=\Big\{h\in L^2(\mathbb{R}^d;dx): \int\limits_{\mathbb{R}^d}|\xi|^\alpha\big|\hat{h}(\xi)\big|^2d\xi<\infty\Big\}
\end{equation*}
as the function with Fourier transform
\begin{equation*}
\mathcal{F}\Big[-(-\Delta)^{\alpha/2}h\Big]=-|\xi|^\alpha\big|\hat{h}(\xi)\big|^2,
\end{equation*}
where $\mathcal{F}[h]=\hat{h}$ represents the usual Fourier transform of the function $h$, see \cite{ChenMeerNane} and the references therein for more on the fractional Laplacian.

 It is a trivial fact that the following integration by parts formula holds for $\phi, \psi\in \text{Dom}\big(\Delta^{\alpha/2}\big)$:
\begin{equation}\label{IntByParts}
\int\limits_{\mathbb{R}^d}\phi(x)\Delta^{\alpha/2}\psi(x)dx=\int\limits_{\mathbb{R}^d}\psi(x)\Delta^{\alpha/2}\phi(x)dx.
\end{equation}
Let $D$ be a bounded open subset of $\mathbb{R}^d.$ Recall that the process $A(t):=X_t$ is a standard spherically symmetric $\alpha$-stable process on $\mathbb{R}^d$ and denote the first exit time of the process $X$ by 
\begin{equation}\label{FirstExiT}
\tau_D=\inf\{t\geq 0: X_t\notin D\}.
\end{equation}
Let $X^D$ denote the process $X$ "killed" upon leaving the domain $D$, i.e
\begin{equation}
\begin{split}
	X_t^D=\begin{cases}
	X_t,& \ \ t<\tau_D\\
	\partial,& \ \ t\geq \tau_D.
	\end{cases}
\end{split}
\end{equation}
Here, $\partial$ is a cemetery point added to $D$. Throughout this paper, we use the convention that any real valued function $f$ can be extended by taking $f(\partial) = 0.$

Then $X^D$ has a jointly continuous transition
density function $p_D(t, x, y)$ with respect to the Lebesgue measure of $D$ \cite{ChenMeerNane, FukOshTak}. Moreover, by the strong Markov property of $X$, one has for $t > 0$ and $x,y\in D,$

\begin{center}
	$p_D(t,x,y)=p(t,x,x)-\mathbb{E}_x\Big[p(t-\tau_D,X_{\tau_D},y); \tau_D<t\Big],$
\end{center}
where $p(t,\cdot,\cdot) $ is the transition density of the "unkilled process" defined earlier. When $\alpha=2,$  $X_t$ corresponds to a Brownian motion $(B_t)_{t\geq 0}$ with variance $2t$ and in this case, $p(t,x,y) $ is explicitly given by 

\begin{center}
	$p(t,x,y)=(4\pi t)^{-d/2}e^{-\frac{|x-y|^2}{4t}}, \ x,y\in \mathbb{R}^d.$
\end{center}

When $\alpha\in(0,2),$ $X_t$ coincides with an $\alpha$-stable L\'evy process given by $X_t=B_{S_t}$ where $(S_t)_{t\geq 0}$ is an $\alpha/2$-stable subordinator with L\'evy measure 
\begin{center}
	$\nu(dx)=\frac{\alpha/2}{\Gamma(1-\alpha/2)}x^{-1-\alpha/2}\boldsymbol{1}_{\{x>0\}}dx.$
\end{center}
No explicit expression is known for $p(t,x,y) $ but the following approximation holds:
\begin{equation}\label{eq:p}
	C_1\min\Bigg(t^{-d/\alpha}, \frac{t}{|x-y|^{\alpha+d}}\Bigg)\leq p(t,x,y)\leq C_2\min\Bigg(t^{-d/\alpha}, \frac{t}{|x-y|^{\alpha+d}}\Bigg)
\end{equation}
for some positive constants $C_1,C_2.$ See for example \cite{GuerNane} and the reference therein.

Denote by $\Big\{p_t^D, t\geq 0\Big\}$ the transition semigroup of $X^D$, i.e,
\begin{equation}\label{Semgrp}
	p_t^Df(x)=\mathbb{E}_x\big[f(X_t^D)\big]=\int\limits_Dp_D(t,x,y)f(y)dy.
\end{equation}
It is well know ( cf. \cite{FukOshTak}) that $u(t, x) = p^D_tf(x)$ is the unique weak solution to 
\begin{equation*}
\begin{cases}
	\frac{\partial u}{\partial t}&=-(-\Delta)^{\alpha/2}u\\
	u(0,x)&=f(x)
\end{cases}
\end{equation*}
on the Hilbert space $L^2(D;dx).$ Therefore, for each $t > 0, p^D_t$
is a Hilbert-Schmidt operator in $L^2(D; dx)$ so it is compact \cite{ChenMeerNane}. Consequently, for the eigenpair defined
in \eqref{eq:Eig}, we have $p^D_t\phi_n =e^{-\lambda_n t}\phi_n$ in $L^2(D; dx)$ for $n \geq 1$ and $t > 0$. 
Recall that the Dirichlet heat kernel $p_D(t,x,y)$ has the spectral decomposition
\begin{equation}\label{eq:SpecDecomp}
	p_D(t,x,y)=\sum\limits_{n=1}^\infty e^{-\lambda_n t}\phi_n(x)\phi_n(y), \ t>0, x,y\in D,
\end{equation}
where $\{\phi_n\}_{n\geq 0}$ is an orthonormal basis of $L^2(D)$ and  $0<\lambda_1\leq \lambda_2\leq \cdots$ is a sequence of positive numbers satisfying

\begin{equation}\label{eq:Eig}
	\begin{split}
		-(-\Delta)^{\alpha/2}\phi_n(x)&=-\lambda_n \phi_n(x), \ x\in D\\
		\phi_n(x)&=0 \ x\in D^c.
	\end{split}
\end{equation}
Note that the series in \eqref{eq:SpecDecomp} converges absolutely and uniformly.
It is well known that 
\begin{equation}\label{eq:eigv}
	c_1n^{\alpha/2}\leq \lambda_n\leq c_2n^{\alpha/d}
\end{equation}
for positive constants $c_1$ and $c_2,$ see for example (again) \cite{GuerNane} and the references within.

Since the eigenfunctions $\{\phi_n\}_{n\geq 0}$ form a complete orthonormal basis, we can write $$f(x)=\sum\limits_{n=1}^\infty \bar{f}(n)\phi_n(x) \ \text{for any}\  f\in L^2(D).$$ Here we denote $\bar{f}(n)=\int\limits_D\phi_n(x)f(x)dx,$  the $\phi_n$-transform of $f.$


\begin{definition}
Following \cite{MeerNaneVella}, a function $h$ is said to be a mild solution of a fractional differential equation if its integral (Laplace or Fourier) transform solves the corresponding algebraic equation.
\end{definition}

\begin{lemma}\cite[Lemma $2.1$]{MeerNaneVella}\label{Lem21}
For any $\lambda>0$, $h(t,x)=\int\limits_0^\infty e^{-\lambda x}g(t,x)dx=\mathbb{E}\big[e^{-\lambda E_t}\big]$ is a mild solution of the distributed-order fractional differential equation 
\begin{equation}\label{DistOrdDEq}
\mathbb{D}^{\nu}h(t,\lambda)=-\lambda h(t,\lambda); \ h(0, \lambda)=1.
\end{equation}
\end{lemma}

\begin{lemma}\cite[Theorem $2.2$]{MeerNaneVella}\label{Thm22}
Let $\nu(d\beta) = p(\beta)d\beta$ for some $p \in C^1(0,1)$, and $0 < \beta_0 < \beta_1 < 1$ be such that
\begin{equation}
	C(\beta_0,\beta_1,p) = \int_{\beta_0}^{\beta_1} \sin(\beta\pi) \Gamma(1-\beta) p(\beta)d\beta > 0.
\end{equation}
Then $h(t,\lambda)$ satisfies $\abs{\partial_t h(t,\lambda)} \leq \lambda k(t)$ where
\begin{equation}\label{eq:k}
	k(t) = \bb{C(\beta_0,\beta_1,p) \pi}^{-1} \bb{\Gamma(1-\beta_1) t^{\beta_1-1} + \Gamma(1-\beta_0) t^{\beta_0-1}}
\end{equation}
and hence is a classical solution. 	
\end{lemma} 

Define the space 
\begin{equation}\label{BanacSp}
C^k(\bar{D})=\big\{u\in C^k(D): D^\gamma u\ \text{is bounded and uniformly continuous for all} \ |\gamma|\leq k \big\}.
\end{equation}
It is well-known that $\big(C^k(\bar{D}), {\|\cdot\|}_{C^k(\bar{D})}\big)$ is a Banach space, where 
\begin{equation}\label{CkNorm}
{\|u\|}_{C^k(\bar{D})}=\sum\limits_{|\gamma|\leq k}\sup\limits_D |D^\gamma u|.
\end{equation}
We will also use the notation 
$${\norm{\cdot}_{2,D}}:={\norm{\cdot}}_{L^2(D)}.$$
	\section{Distributed-order space time fractional diffusion equations}
	
	In this section, we provide classical solutions to the distributed-order space-time fractional diffusion equation $\mathbb{D}^{(\nu)}u=-(-\Delta)^{\alpha/2}u$ in the bounded domains $D\subset \mathbb{R}^d.$

Let $D_\infty=(0,\infty)\times D$ and define
\begin{center}
	
	$\mathcal{H}_{\Delta^{\alpha/2}}(D_\infty)=\big\{u:D_\infty\rightarrow\mathbb{R}: \ \Delta^{\alpha/2}u\in C(D_\infty), \ |\partial_tu(t,x)|\leq k(t)g(x), \ g\in L^{\infty}(D), t>0\},$ where $k$ is given by \eqref{eq:k}.
\end{center}


We are now ready to state our first main result. 
\begin{theorem}\label{tm31}
Let $D $ be a regular   open  set in $\mathbb{R}^d$ and $P_t^D$ be the semigroup of a killed  $\alpha$-stable process $\{X(t)\}$ on $D$.  Let $E_t = \inf\{\tau: W_\tau>t\}$ be the inverse of the subordiator $W_t$, independent of $\{X(t)\}$, with L\'evy measure given by \eqref{eq:Lmeas}.  Suppose that $\mu(d\beta) = p(\beta)d\beta$, as in Lemma \ref{Thm22},	and $\mathbb{D}^{(\nu)}$ is the distributed-order fractional derivative defined by \eqref{eq:dofd1}.  Then, for any $f \in \text{Dom}\big((\Delta)^{\alpha/2}_D\big) \cap
C^1(\bar{D}) \cap C^2(D)$ for which the eigenvalue expansion of $-(-\Delta)^{\alpha/2} f$ with respect to the complete orthonormal basis $\{\phi_n:n \in \N\}$ converges uniformly and absolutely, the unique classical solution of the distributed order space-time fractional diffusion equation
\begin{equation}\label{eq:theorem}
\begin{split}
	\mathbb{D}^{(\nu)} u(t,x) =& -(-\Delta)^{\alpha/2} u(t,x), \ t>0 , \ x\in D, \\
u(t,x) =& 0,  \ x \in D^C,\ t>0,\\
u(0,x) =& f(x), \ x \in D,
\end{split}
\end{equation}

for $u \in \H_{\Delta^{\alpha/2}}(D_\infty) \cap C_b(\bar{D}_\infty) \cap C^1(\bar{D})$, is given by
\begin{align}\label{eq:expectation}
	u(t,x) &= \mathbb{E}_x[f(X(E_t))I(\tau_D(X)>E_t)] \nonumber\\
	&= \sum_{n=1}^\infty \bar{f}(n)\phi_n(x)h(t,\lambda_n), 
\end{align}
where $\displaystyle h(t,\lambda) = \mathbb{E}(e^{-\lambda E_t}) $ is the Laplace transform of $E_t$.
\end{theorem}  
%
%

\begin{proof}
The proof follows closely ideas from \cite{MeerNaneVella}.  Assume that $u(t,x)$ solves \eqref{eq:theorem}. Using \eqref{IntByParts}
, we have
%
%
\begin{align*}
\int_D \phi_n(x)\Delta^{\alpha/2}u(t,x)dx =& \int_D u(t,x)\Delta^{\alpha/2}\phi_n(x)dx\\
=& -\lambda_n \int_D u(t,x)\phi_n(x)dx\\
=& -\lambda_n \bar{u}(t,x).
\end{align*}
Next, using \eqref{eq:caputo}, and \eqref{eq:dofd1}, observe that
\begin{align}
	\int_D \phi_n(x)\D^{(\nu)}u(t,x)dx &= \int_D \phi_n(x) \int_0^1 \frac{\partial^\beta}{\partial t^\beta}u(t,x)\Gamma(1-\beta)p(\beta)d\beta dx \nonumber\\
	&= \int_D \phi_n(x) \int_0^1 \frac{1}{\Gamma(1-\beta)} \int_0^t \frac{\partial u(s,x)}{\partial s} \frac{ds}{(t-s)^\beta}\Gamma(1-\beta)p(\beta)d\beta dx \nonumber\\
	&= \int_D \phi_n(x) \int_0^1 \int_0^t \frac{\partial u(s,x)}{\partial s}\frac{ds}{(t-s)^\beta}p(\beta)d\beta dx \nonumber\\
	&= \int_0^1 \int_0^t \Bigg(\int_D \phi_n(x)\frac{\partial}{\partial s}u(s,x)dx\Bigg) \frac{ds}{(t-s)^\beta}p(\beta)d\beta \nonumber\\
	&= \int_0^1 \int_0^t \frac{\partial}{\partial s} \Bigg(\int_D \phi_n(x)u(s,x)dx\Bigg) \frac{ds}{(t-s)^\beta}p(\beta)d\beta \nonumber\\
	&= \int_0^1 \frac{1}{\Gamma(1-\beta)} \int_0^t \frac{\partial}{\partial s}\bar{u}(s,n)\frac{ds}{(t-s)^\beta}\Gamma(1-\beta)p(\beta)d\beta \nonumber\\
	&= \D^{(\nu)}\bar{u}(t,n).
\end{align}
The Fubini-Tonelli argument used in the last equality above can be justified as in \cite[Page 223]{MeerNaneVella}.
Now applying the $\phi_n$-transform to both sides to \eqref{eq:theorem}, we have that
\begin{equation} \label{eq:Du_bar}
	\D^{(\nu)}\bar{u}(t,n) = -\lambda_n\bar{u}(t,x).
\end{equation}
Since $u$ is uniformly continuous on $C\big([0,\epsilon]\times \bar{D}\big)$ it is uniformly bounded on $[0,\epsilon] \times D$ for any $\epsilon > 0$.  Thus, by the dominated convergence theorem,  we have 
%
%
it follows that
\begin{align}
	\lim\limits_{t \to 0} \int\limits_D u(t,x)\phi_n(x)dx = 
	 \bar{f}(n),
\end{align}
and further $\bar{u}(0,n) = \bar{f}(n)$.  A similar argument shows that the function $t\mapsto\bar{u}(t,n)$ is a continuous of $t\in[0,\infty)$ for every $n.$
Next denote the Laplace transform and the $\phi_n$-Laplace transform of $u$ respectively by 
\begin{center}
	$\tilde{u}(s,x)=\int_0^\infty e^{-st}u(t,x)dt$ and $\hat{u}(s,n)=\int\limits_D\phi_n(x)\tilde{u}(s,x)dx$.
\end{center}
Apply now the Laplace transform to both sides of \eqref{eq:Du_bar} and using \eqref{eq:caputo laplace} we get
\begin{equation}
	\int_0^1 \Big[s^\beta \hat{u}(s,n) - s^{\beta-1} \bar{u}(0,n)\Big] \Gamma(1-\beta) p(\beta)d\beta = -\lambda_n \hat{u}(s,n),
\end{equation}
and further
\begin{equation}
	\hat{u}(s,n) = \frac{\bar{f}(n) \int_0^1 s^{\beta-1} \Gamma(1-\beta) p(\beta)d\beta}{\int_0^1 s^\beta \Gamma(1-\beta) p(\beta)d\beta + \lambda_n}.
\end{equation}
Then, by \eqref{eq:psi_w}, 
\begin{align}\label{eq:u_hat}
	\hat{u}(s,n) &= \frac{\overline{f}(n)\psi_W(s)}{s(\psi_W(s)+\lambda_n} \nonumber\\
	&= \frac{1}{s}\overline{f}(n)\psi_W(s) \int_0^\infty e^{-(\psi_W(s)+\lambda_n)l)}dl \nonumber\\
	&= \int_0^\infty e^{-\lambda_n l}\overline{f}(n)\frac{1}{s}e^{-l\psi_W(s)}dl, 
\end{align}
by properties of the exponential density.  Using that $\{\phi_n:\ n \in \N\}$ is a complete orthonormal basis of $L^2(D)$, we calculate the $\phi_n$-transform of the killed semigroup $P_l^Df(x) = \sum\limits_{m=1}^\infty e^{-\lambda_m l}\phi_m(x)\overline{f}(m)$ as
\begin{align}\label{eq:Tf_phitransform}
	\overline{[P_l^Df]}(m) &= \int_D \phi_n(x)P_l^Df(x)dx \nonumber\\
	&= \int_D \phi_n(x) \int_D p_D(t,x,y)f(y)dydx \nonumber\\
	&= \int_D \phi_n(x) \int_D \sum_{m=1}^\infty e^{-\lambda_m l}\phi_m(x)\phi_m(y)f(y)dydx \nonumber\\
	&= \int_D \phi_n(x) \sum_{m=1}^\infty e^{-\lambda_m l}\phi_m(x) \int_D \phi_m(y)f(y)dydx \nonumber\\
	&= \int_D \phi_n(x) \sum_{m=1}^\infty e^{-\lambda_m l}\phi_m(x)\overline{f}(m)dx \nonumber\\
	&= \sum_{m=1}^\infty e^{-\lambda_m l}\overline{f}(m) \int_D \phi_n(x)\phi_m(x)dx \nonumber\\
	&= e^{-\lambda_m l}\overline{f}(m).
\end{align}
Since $P_t^D$ is a contraction semigroup on $L^2(D)$, $P_l^Df \in L^2(D)$, so Fubini-Tonelli applies.  By \cite[(3.20)]{MeerSchef1}, it follows that

\begin{equation}\label{eq:psi*e}
	\frac{1}{s}\psi_W(s)e^{-\psi_W(s)l} = \int_0^\infty e^{-st}g(t,l)dt,
\end{equation}
where $g(t,l)$ is the smooth density of $E_t$.  From \eqref{eq:u_hat}, \eqref{eq:Tf_phitransform}, and \eqref{eq:psi*e}, we have that
\begin{align*}
	\int_0^\infty e^{-st}\overline{u}(t,n)dt &= \hat{u}(s,n)\\
	&= \int_0^\infty \overline{[p_l^Df]}(n) \bb{\int_0^\infty e^{-st}g(t,l)dt}dl\\
	&= \int_0^\infty e^{-st} \bb{\int_0^\infty \overline{[P_l^Df]}(n)g(t,l)dl}dt.
\end{align*}
Then, by the uniqueness of the Laplace transform and \eqref{eq:Tf_phitransform},
\begin{align}\label{eq:u_bar}
	\overline{u}(t,n) &= \int_0^\infty \overline{[P_l^Df]}(n)g(t,l)dl \nonumber\\
	&= \overline{f}(n) \int_0^\infty e^{-\lambda_n l}dl \nonumber\\
	&= \overline{f}(n)h(t,\lambda_n),
\end{align}
recalling that $h(t,\lambda)$ is the Laplace transform of $E_t$.  Therefore, by inverting 
%
%
the $\phi_n$-transform $\overline{u}(t,n)$ in \eqref{eq:u_bar}, an $L^2$-convergent solution to \eqref{eq:theorem} is
\begin{align}\label{eq:solution}
	u(t,x) &= \sum_{n=1}^\infty \overline{u}(t,n)\phi_n(x) \nonumber\\
	&= \sum_{n=1}^\infty \overline{f}(n)\phi_n(x)h(t,\lambda_n)
\end{align}
for $t \geq 0$.  As such, the remainder of the proof will be to show that \eqref{eq:solution} converges pointwise and satisfies the conditions in \eqref{eq:theorem}. This can be done following the 8-step proof provided in \cite{MeerNaneVella} with some modifications.

\vspace{0.25cm}
%

\begin{enumerate}
\item[\textbf{Step 1.}] The series in \eqref{eq:solution} converges uniformly in $t\in[0,\infty) $ in the $L^2$- sense.

\vspace{0.25cm}
 The proof in this step is  very similar to \textbf{Step 1} in Theorem \ref{eq:Thm2} (ahead), so we omit the details here to avoid repetitiveness.

\vspace{0.25cm}

\item[\textbf{Step 2.}] The function $t \to u(t,\cdot) \in C\big({(0,\infty);L^2(D)}\big)$ and $u(t, \cdot)$ takes the initial datum $f$ in the sense of $L^2(D).$

\vspace{0.25cm}
Using $\{E_t \leq x\} = \{W_x \geq t\}$, 
it is easy to check that $E_t \to E_0 \equiv 0$ in distribution as $t \to 0^+$, and hence the Laplace transforms converge: $h(t,\lambda_n) \to 1$ as $t \to 0$.

Now, since $f\in L^2(D)$, given $\epsilon>0,$ we can choose $n_0(\epsilon)$ such that
\begin{equation}\label{eq:L2Normf}
\sum\limits_{n=n_0(\epsilon)}^\infty \big(\bar{f}(n)\big)^2<\epsilon.
\end{equation}
Thus, fixing $\epsilon\in(0,1)$ and choosing $n_0=n_0(\epsilon)$ as in \eqref{eq:L2Normf}, we have
\begin{align*}
	{\big\|u(t,x)-f\big\|}_{2,D}^2=&{\sum\limits_{n=1}^\infty\big(\bar{f}(n)\big)^2\big(h(t,\lambda_n)-1\big)^2}\\
	\leq & \sum\limits_{n=1}^{n_0}\big(\bar{f}(n)\big)^2\big(h(t,\lambda_n)-1\big)^2\\
	&\qquad+ \sum\limits_{n=n_0+1}^\infty\big(\bar{f}(n)\big)^2\big(h(t,\lambda_n)-1\big)^2\\
	\leq & \big(h(t,\lambda_{n_0})-1\big)^2\Big({\big\|f\big\|}_{2,D}^2+\epsilon\Big).
\end{align*}
 It follows that  $\norm{u(t,\cdot)-f}_{2,D} \to 0$, as $t \to 0$ since $h(t,\lambda_n) \to 1$ as $t \to 0$.
 Finally,  the continuity of $t \to u(t,\cdot)$ in $L^2(D)$ at every point $t \in (0,\infty)$ follows by a similar argument.

\vspace{0.25cm}

\item[\textbf{Step 3.}] The following decay estimate holds:  $\norm{u(t,\cdot)}_{2,D} \leq h(t,\lambda_1)\norm{f}_{2,D}$

The proof is a mere application of   Parseval’s identity and the fact that  $\lambda_n$ is increasing in $n$, and $h(t,\lambda_n)$ is non-increasing for $n \geq 1$,.

\vspace{0.25cm}

\item[\textbf{Step 4.}]  The series in \eqref{eq:solution} converges uniformly and absolutely.

\vspace{0.25cm}
Again the proof in this step is  very similar to \textbf{Step 1} in Theorem \ref{eq:Thm2} (ahead), so the details are omitted here.

\vspace{0.25cm}

\item[\textbf{Step 5.}]  $\D^{(\nu)}$ and $\Delta^{{\alpha}/{2}}$ can be applied term by term in \eqref{eq:solution} to show that \eqref{eq:solution} is a strong solution to \eqref{eq:theorem} and $u\in \H_{\Delta^{\alpha/2}}(D_\infty) \cap C_b(\overline{D}_\infty)$.

\vspace{0.25cm}
Since the eigenfunction expansion of $\Delta^{{\alpha}/{2}} f$ converges absolutely and uniformly, the series $\sum\limits_{n=1}^\infty \overline{f}(n)h(t,\lambda_n)\Delta^{{\alpha}/{2}}\phi_n(x)$ is absolutely convergent in $L^\infty(D)$ uniformly in $t\in[0,\infty)$. Moreover, since $\D^{(\nu)}h(t,\lambda) = -\lambda h(t,\lambda)$, it follows that
\[ \sum_{n=1}^\infty \overline{f}(n)\phi_n(x)\D^{(\nu)}h(t,\lambda_n) = \sum_{n=1}^\infty \bar{f}(n)h(t,\lambda_n)\Delta^{{\alpha}/{2}}\phi_n(x), \] 
where both series converge absolutely and uniformly. This shows that  $\D^{(\nu)}$ and $\Delta^{{\alpha}/{2}}$ can be applied term by term in \eqref{eq:solution} to show that \eqref{eq:solution} is a strong solution to \eqref{eq:theorem}. Since $\Delta^{{\alpha}/{2}} f$ has an absolutely and uniformly convergent series expansion with
respect to $(\phi_n)$, it follows from Lemma \ref{Thm22}
that $u \in \H_{\Delta^{\alpha/2}}(D_\infty) \cap C_b(\overline{D}_\infty)$.

\vspace{0.25cm}

\item[\textbf{Step 6.}] $u\in C^1(\overline{D}).$

This follows easily from  \eqref{CkNorm} and the absolute and uniform convergence of the series defining $f$ together with the fact that $f\in C^1(\bar{D})$.

\vspace{0.25cm}

\item[\textbf{Step 7.}] $u(t,x) = \mathbb{E}_x[f(X(E_t))I(\tau_D(X)>E_t)].$

\vspace{0.25cm}

 By \eqref{eq:Tf_phitransform} and \eqref{eq:solution}, a direct calculation shows that
\begin{align}
	u(t,x) &= \sum_{n=1}^\infty \phi_n(x) \int_0^\infty \overline{[T_D(l)f]}(n)g(t,l)dl \nonumber\\
	&= \int_0^\infty \bb{\sum_{n=1}^\infty \phi_n(x)\overline{f}(n)e^{-l\lambda_n}}g(t,l)dl \nonumber\\
	&= \int_0^\infty T_D(l)f(x)g(t,l)dl \nonumber\\
	&= \mathbb{E}_x\bb{f\big({X(E_t)}\big)I\big({\tau_D(X)>E_t}\big)},
\end{align}
which proves \eqref{eq:expectation}.
\vspace{0.25 cm}

\item[\textbf{Step 8.}] The function $u$ defined by the series \eqref{eq:solution} is the unique solution to \eqref{eq:theorem}.

 To prove this, it is enough to show that if   $u_i$, $i = 1,2$, are two solutions of \eqref{eq:theorem} with the same initial data $u_i(0,x) = f (x)$, then $U = u_1-u_2\equiv0$.  But this fact is very easy to check since $U$ is also a solution with the corresponding $f \equiv 0$ and  thus  $U(t,x) = 0$ for all $(t,x) \in [0,\infty) \times D$.
\end{enumerate}
\end{proof}

%
%


The next result provides sufficient conditions for Theorem \eqref{tm31} to hold.
\begin{corollary}
 Let $f \in C_c^{2 k}(D)$ be an $2 k$-times continuously differentiable function with compact support in $D$.  If $k > 1+3d/2\alpha$, then \eqref{eq:theorem} has a classical (strong) solution.  In particular, if $f \in C_c^\infty(D)$, then the solution of \eqref{eq:theorem} is in $C^\infty(D)$.
\end{corollary} 

\begin{proof}
Using \eqref{eq:SpecDecomp} together with \eqref{eq:p}, we have
\begin{center}
	$e^{-\lambda_n t}|\phi_n(x)|^2\leq \sum\limits_{k=1}^\infty e^{-\lambda_k t}|\phi_k(x)|^2=p_D(t,x,x)\leq p(t, x,x)\leq C_1t^{-d/\alpha}. $
\end{center}
Hence,
$|\phi_n(x)|\leq C_2e^{\lambda_nt/2}t^{-d/2\alpha}$. Thus, setting $t=\lambda_n^{-1}$ yields the estimate
\begin{equation}\label{eq:phiest}
|\phi_n(x)|\leq C_3\lambda_n^{d/2\alpha}.
\end{equation}
	  Applying \eqref{IntByParts} $k$-times yields
	\begin{equation} \label{eq:D^k bar}
		\overline{\Big(\Delta^{\alpha/2}\Big)^k f}(n) = \int_D \Big(\Delta^{\alpha/2}\Big)^kf(x)\phi_n(x)dx = (-\lambda_n)^k\overline{f}(n),
	\end{equation}
	where then by the Cauchy-Schwarz inequality
	\begin{align*}
	\int_D \big(\Delta^{\alpha/2}\big)^kf(x)\phi_n(x)dx \leq& \Bigg(\int_D \Big[\big(\Delta^{\alpha/2}\big)^k f(x)\Big]^2dx\Bigg)^{1/2} \Bigg(\int_D \Big[\phi_n(x)\Big]^2dx\Bigg)^{1/2}\\
	=& \Bigg(\int_D \Big[\big(\Delta^{\alpha/2}\big)^k f(x)\Big]^2dx\Bigg)^{1/2}\\
	=&c_k
	\end{align*}
	for some constant $c_k>0$.    As such, by \eqref{eq:D^k bar}, $\displaystyle \abs{\overline{f}(n)} \leq c_k{\lambda_n}^{-k}$.  Since
	\begin{equation}\label{SerieFractLap}
	\Delta^{\alpha/2} f(x) = \sum_{n=1}^\infty -\lambda_n\overline{f}(n)\phi_n(x),
	\end{equation}

	to get the absolute and uniform convergence of the series in \eqref{SerieFractLap}, we consider
	\begin{align*}
		\sum_{n=1}^\infty \lambda_n\abs{\phi_n(x)}\abs{\overline{f}(n)} &\leq \sum_{n=1}^\infty \lambda_n\pp{\lambda_n}^{d/2\alpha}c_k\lambda_n^{-k}\\
		&\leq c_k \sum_{n=1}^\infty \Big(n^{\alpha/d}\Big)^{1+d/2\alpha-k}\\
		&= C_k \sum_{n=1}^\infty n^{\alpha/d+1/2-k\alpha/d}.
	\end{align*}
	This  series converges when $\alpha/d+1/2-k\alpha/d<-1$, and hence $k >1+3d/2\alpha.$ Note the use of \eqref{eq:eigv} and \eqref{eq:phiest} in the lines above.
\end{proof}

\section{Multi-term space-time fractional diffusion equations}
We now allow the mixture measure $\nu$ to contain atoms. Let $0<\beta_1<\beta_2<\cdots<\beta_n<1$, $c_j>0, \ j=1,\cdots, n$ and define 
\begin{equation}\label{nuDisc}
	\mu(d\beta)= \sum\limits_{j=1}^n c_j^{\beta_j}\Big(\Gamma(1-\beta_j)\Big)^{-1}\delta(\beta-\beta_j)d\beta,
\end{equation}
where $\delta(\cdot)$ is the Delta Dirac measure. Note using \eqref{nuDisc} in \eqref{eq:MixM} and \eqref{eq:dofd1} leads to the so-called multi-term time- fractional differential operator $\sum\limits_{j=1}^\infty c_j^{\beta_j}\frac{\partial^{\beta_j}}{\partial t^{\beta_j}}$. Even though this operator is a special case of the distributed-order derivative $\mathbb{D}^{(\nu)}$, the asymptotic behavior of solutions to initial-value problems with these two operators behave very differently \cite{LiKoTo}. We refer the interested reader to  \cite{LiLiuMoto,LiMoto,MijNane} for more on the multi-term time-fractional diffusion equations. 

Additionally, define the subordinator

\begin{equation}\label{SubW}
	W_t=\sum\limits_{j=1}^n c_j W_t^{\beta_j},
\end{equation}
where the $W_t^{\beta_j}$'s are independent subordinators. For simplicity's sake, we use the notation $\partial_t^\beta:=\frac{\partial^{\beta}}{\partial t^{\beta}}$ in the remainder of this article.

\begin{lemma}\label{lm21MfD}
	For any $\lambda>0,$ $h(t,\lambda)=\mathbb{E}(e^{-\lambda E_t})$ is a mild solution of the multi-term time fractional differential equation 
	\begin{equation}
		\sum\limits_{j=1}^n c_j^{\beta_j}\partial_t^{\beta_j}h(t,\lambda)=-\lambda h(t,\lambda), \ h(0,\lambda)=1.
	\end{equation}
\end{lemma}
\begin{proof}
	Clearly, $h(0,\lambda)=1.$
	
	Next, using \eqref{eq:caputo laplace}, we have
	
	\begin{align*}
		\int\limits_0^\infty e^{-st}\sum\limits_{j=1}^n c_j^{\beta_j}\partial_t^{\beta_j}h(t,\lambda) dt=& \sum\limits_{j=1}^n c_j^{\beta_j}\int\limits_0^\infty e^{-st}\partial_t^{\beta_j}h(t,\lambda) dt\\
		=& \sum\limits_{j=1}^n c_j^{\beta_j}\Big(s^{\beta_j}\tilde{h}(s,\lambda)-s^{\beta_j-1}\Big).
	\end{align*}
	On the other hand, using \eqref{eq:Dh laplace} together with the fact that $\psi_W(s)=\sum\limits_{j=1}^n c_j^{\beta_j}s^{\beta_j}$, it follows that
	\begin{align*}
		\lambda \tilde{h}(s,\lambda)=&\Big(\frac{1}{s}-\tilde{h}(s,\lambda)\Big)\psi_W(s)\\
		=-&\sum\limits_{j=1}^n c_j^{\beta_j}\partial_s^{\beta_j}\tilde{h}(s,\lambda)
	\end{align*}
	Since $E_t$ has continuous paths, the dominated convergence theorem implies that $t\mapsto h(t,\lambda)$ is a continuous function, then the uniqueness of the Laplace transform concludes the proof.
\end{proof}
The next result shows that we can relax the regularity condition on the mixture measure in Lemma \ref{Thm22}.
\begin{lemma}\label{lmConvh}
Let the weight measure be given by \eqref{nuDisc}. Then
\begin{equation}\label{delKej}
\big|\partial_th(t,\lambda)\big|\leq \lambda k_e(t), \ \text{for all} \ j=1,\cdots,n,
\end{equation}
where
\begin{equation}\label{Ke}
k_e(t)=\Big(c_j^{\beta_j}\sin(\beta_j\pi)\Big)^{-1}t^{\beta_j-1}, \ j=1,\cdots,n. 
\end{equation} 
\end{lemma}
\begin{proof}
Using $(2.19)$ in \cite{Koch} which follows from inverting the Laplace transform of $h(t,\lambda)$, we have
$h(t,\lambda)=-\frac{\lambda}{\pi}\int\limits_0^\infty r^{-1}e^{-tr}\phi(r,1)dr,$ where

\begin{align*}
\phi(r,1)=\frac{ \sum\limits_{j=1}^n c_j^{\beta_j}\big(\Gamma(1-\beta_j)\big)^{-1}r^{\beta_j}\sin(\beta_j \pi)}{\Bigg[\sum\limits_{j=1}^n c_j^{\beta_j}\Big(\Gamma(1-\beta_j)\Big)^{-1}r^{\beta_j}\cos(\beta_j \pi)+\lambda\Bigg]^2+\Bigg[\sum\limits_{j=1}^n c_j^{\beta_j}\Big(\Gamma(1-\beta_j)\Big)^{-1}r^{\beta_j}\sin(\beta_j \pi)\Bigg]^2}.
\end{align*}
Next, a direct calculation shows that 

\begin{align*}
\big|\partial_th(t,\lambda)\big|=& \frac{\lambda}{\pi}\int\limits_0^\infty e^{-tr}\frac{ \sum\limits_{j=1}^n c_j^{\beta_j}\big(\Gamma(1-\beta_j)\big)^{-1}r^{\beta_j}\sin(\beta_j \pi)}{\Bigg[\sum\limits_{j=1}^n c_j^{\beta_j}\Big(\Gamma(1-\beta_j)\Big)^{-1}r^{\beta_j}\cos(\beta_j \pi)+\lambda\Bigg]^2+\Bigg[\sum\limits_{j=1}^n c_j^{\beta_j}\Big(\Gamma(1-\beta_j)\Big)^{-1}r^{\beta_j}\sin(\beta_j \pi)\Bigg]^2}dr \\
\leq & \frac{\lambda}{\pi}\Big(c_j^{\beta_j}\sin(\beta_j \pi)\Big)^{-1} \Gamma(1-\beta_j)\int\limits_0^\infty e^{-tr}r^{-\beta_j} dr.
\end{align*}
Finally, the definition of the Euler's gamma function concludes the proof.
\end{proof}
\begin{remark}
Since $\Big|\sum\limits_{j=1}^n c_j^{\beta_j}\partial_t^{\beta_j}h(t,\lambda)\Big|\leq \sum\limits_{j=1}^n c_j^{\beta_j}\big|\partial_t^{\beta_j}h(t,\lambda)\big|$, it follows from \eqref{eq:caputo} and Lemma \ref{lmConvh} that $\sum\limits_{j=1}^n c_j^{\beta_j}\partial_t^{\beta_j}h(t,\lambda)$  exists and $h(t,\lambda)$ is an eigenfunction in the strong sense.

\end{remark}

Define \begin{center}
	$\mathcal{H}^e_{\Delta^{\alpha/2}}(D_\infty)=\big\{u:D_\infty\rightarrow\mathbb{R}: \ \Delta^{\alpha/2}u\in C(D_\infty), \ |\partial_tu(t,x)|\leq k_e(t)g(x), \ g\in L^{\infty}(D), t>0\}, $
\end{center}

where $k_e$ is given by \eqref{Ke}. 

We are now ready to state our main result in this section.

\begin{theorem}\label{eq:Thm2}
Let $X(t)$ be an $\alpha$-stable process. Then under the conditions of Theorem \ref{tm31}  , for any $f\in \text{Dom}\big((\Delta)^{\alpha/2}_D\big)\cap C^1(\bar{D})\cap C^2(D),$ the classical solution of 
\begin{equation}\label{eq:theorem2}
	\begin{split}
		\sum\limits_{j=1}^n c_j^{\beta_j}\partial_t^{\beta_j} u(t,x) =& -(-\Delta)^{\alpha/2} u(t,x), \ t>0 , \ x\in D, \\
		u(t,x) =& 0,  \ x \in D^C,\ t>0,\\
		u(0,x) =& f(x), \ x \in D,
	\end{split}
\end{equation}

for $u \in \H^e_{\Delta^{\alpha/2}}(D_\infty) \cap C_b(\bar{D}_\infty) \cap C^1(\bar{D})$, is given by
\begin{align}\label{eq:expect}
	u(t,x) &= \mathbb{E}_x[f(X(E_t))I(\tau_D(X(E))>t)] \nonumber\\
	&= \sum_{n=1}^\infty \bar{f}(n)\phi_n(x)h(t,\lambda_n), 
\end{align}
where $\displaystyle h(t,\lambda)$  is again the Laplace transform of $E_t$.
\end{theorem}
\begin{proof}
The proof uses a separation of variables scheme and follows similar ideas from \cite{MeerNaneVella}. So, suppose $u(t,x)=G(t)F(x)$ solves \eqref{eq:theorem2}, then

$$ \frac{\sum\limits_{j=1}^n c_j^{\beta_j}\partial_t^{\beta_j}G(t) }{G(t)}=\frac{-(-\Delta)^{\alpha/2}F(x)}{F(x)}=-\mu.$$

We first consider the eigenvalue problem

\begin{equation}
\begin{split}
-(-\Delta)^{\alpha/2}F(x)=&-\mu F(x), \ x\in D\\
F(x)=&0, \ x\in D^c	,
\end{split}
\end{equation}
which is solved by the eigenpair $(\lambda_n, \phi_n)$. Also by Lemma \ref{lm21MfD}, the equation \begin{center}
	$\sum\limits_{j=1}^n c_j^{\beta_j}\partial_t^{\beta_j}G(t)=-\mu G(t)$
\end{center} is solved by $G(t)=\bar{f}(n)h(t,\lambda_n)$ with $\bar{f}(n)=u(0,n).$

Next, define the sequence of functions
\begin{equation}\label{eq:uN}
u_N(t,x)=\sum\limits_{n=1}^N\bar{f}(n)\phi_n(x)h(t,\lambda_n).
\end{equation}
\begin{enumerate}
\item[\textbf{Step 1.}]
 The sequence defined in \eqref{eq:uN} is a Cauchy sequence in $L^2(D)\cap L^\infty(D)$ uniformly for $t\in[0, \infty)$.

Fix $\epsilon>0.$ Since $f\in L^2(D)$, we can choose $n_0(\epsilon)$ such that $\sum\limits_{n=n_0(\epsilon)}^\infty\big(\bar{f}(n)\big)^2<\epsilon.$

Now for $n>M>n_0(\epsilon)$ and $t\geq 0,$
\begin{align*}
{\big\|u_N(t,x)-u_M(t,x)\big\|}_{2,D}^2=&{\Bigg\|\sum\limits_{n=M}^N\bar{f}(n)\phi_n(x)h(t,\lambda_n)\Bigg\|}_{2,D}^2\\
\leq & {\big\|f\big\|}_{2,D}^2\\
<&\epsilon.
\end{align*}
Thus, the series in \eqref{eq:uN} converges uniformly in $L^2(D)$ for $t\in[0,\infty).$

Next, since the series defined in \eqref{SerieFractLap} is absolutely and uniformly convergent, there exists an $m_0(\epsilon)$ such that for all $x\in D,$

\begin{equation}
\sum\limits_{n=m_0(\epsilon)}^\infty|\bar{f}(n)\phi_n(x)|\leq \sum\limits_{n=m_0(\epsilon)}^\infty\lambda_n|\bar{f}(n)\phi_n(x)|<\epsilon.
\end{equation}
It follows that for $N>M>m_0(\epsilon), t\geq 0$ and $x\in D,$  we have 

\begin{align*}
	{\Big|u_N(t,x)-u_M(t,x)\Big|}=&{\Bigg|\sum\limits_{n=M}^N\bar{f}(n)\phi_n(x)h(t,\lambda_n)\Bigg|}\\
	\leq &{\sum\limits_{n=M}^N\big|\bar{f}(n)\phi_n(x)\big|}\\
	<&\epsilon.
\end{align*}
This shows that the sequence defined in \eqref{eq:uN} is also a Cauchy sequence in $L^\infty(D).$ Consequently, the series defined by 
\begin{equation}\label{eqU}
u(t,x)=\sum\limits_{n=1}^\infty \bar{f}(n)\phi_n(x)h(t,\lambda_n) 
\end{equation}
is absolutely and uniformly convergent. 
This series satisfies the initial-exterior conditions in \eqref{eq:theorem2}.

\item[\textbf{Step 2.}] The multi-term time fractional derivative $\sum\limits_{j=1}^n c_j^{\beta_j}\partial_t^{\beta_j}$ and the fractional Laplacian $-(-\Delta)^{\alpha/2}$ can be applied term by term in \eqref{eq:theorem2}.

\vspace{0.25cm}
Note that 
\begin{align*}
{\Bigg|\sum\limits_{n=1}^\infty\bar{f}(n)h(t,\lambda_n)\Delta^{\alpha/2}\phi_n(x)\Bigg|}=&{\Bigg|\sum\limits_{n=1}^\infty\lambda_n\bar{f}(n)h(t,\lambda_n)\phi_n(x)\Bigg|}\\
\leq &{\sum\limits_{n=1}^\infty\lambda_n\big|\bar{f}(n)\phi_n(x)\big|}\\
<&\infty.
\end{align*}
Thus, the series $\sum\limits_{n=1}^\infty\bar{f}(n)h(t,\lambda_n)\Delta^{\alpha/2}\phi_n(x)$ is absolutely convergent in $L^2(D)$ uniformly in $t\geq 0.$ Moreover, by Lemma \ref{lm21MfD} we have
\begin{align*}
	\sum\limits_{n=1}^\infty\bar{f}(n)\phi_n(x)\Bigg(\sum\limits_{j=1}^n c_j^{\beta_j}\partial_t^{\beta_j}h(t,\lambda_n)\Bigg)=&\sum\limits_{n=1}^\infty\bar{f}(n)h(t,\lambda_n)\Delta^{\alpha/2}\phi_n(x).\\
\end{align*}
Finally,

\begin{align*}
\sum\limits_{j=1}^n c_j^{\beta_j}\partial_t^{\beta_j}u(t,x)=& \sum\limits_{j=1}^n c_j^{\beta_j}\partial_t^{\beta_j}\Bigg(\sum\limits_{n=1}^\infty\bar{f}(n)\phi_n(x)h(t,\lambda_n)\Bigg)\\
=& \sum\limits_{n=1}^\infty\bar{f}(n)\phi_n(x)\Bigg(\sum\limits_{j=1}^n c_j^{\beta_j}\partial_t^{\beta_j}h(t,\lambda_n)\Bigg)\\
=&-(-\Delta)^{\alpha/2}u(t,x).
\end{align*}
Therefore, we conclude that the series \eqref{eqU} is a classical solution of problem \eqref{eq:theorem2}.
\end{enumerate}
\vspace{0.25cm}

The rest of the proof follows the same arguments in \textbf{Steps 2, 5, 6} and \textbf{8} in the proof of Theorem \ref{tm31}. 
\end{proof}
\begin{remark}
Note that by \cite[corollary 3.2]{MeerNaneVella2}, $$\mathbb{E}_x[f(X(E_t))I(\tau_D(X(E))>t)]=\mathbb{E}_x\bb{f\big({X(E_t)}\big)I\big({\tau_D(X)>E_t}\big)}.$$ This shows that the stochastic solutions provided in Theorems \ref{tm31} and \ref{eq:Thm2} are equivalent.
\end{remark}

\section{Acknowlegment}		
The authors thank E. Nane for useful suggestions that has improved the quality of this paper. N. Guerngar would also like to thank E. Nane for inviting him in August 2022 for  a week of work where meaningful discussions have led to the creation of a new paper.
	

\end{document}